\documentclass[a4paper]{mathscan}
\usepackage{amsfonts,amsmath,amssymb,amsthm, geometry}
\usepackage{graphicx} 
\usepackage{srcltx}
\usepackage[all]{xy}

\newtheorem{thm}{Theorem}[section] 
       \newtheorem{pro}[thm]{Proposition} 
       \newtheorem{cor}[thm]{Corollary}    
       \newtheorem{lem}[thm]{Lemma}        

       \theoremstyle{definition} 
         
       \newtheorem{rem}[thm]{Remark}   
       \newtheorem{defn}[thm]{Definition}  
       \newtheorem{exam}[thm]{Example}

\renewcommand{\Box}{\square}    

\newcommand{\codim}{{\rm{codim\hspace{2pt}}}}
\newcommand{\diffeo}{{\rm{diffeo}}}

\newcommand{\proj}{{\rm{proj}}}

\newcommand{\Sing}{{\rm{Sing\hspace{2pt}}}}

\newcommand{\rank}{{\rm{rank\hspace{2pt}}}}

\newcommand{\re}{{\rm{Re}}}

\newcommand{\im}{{\rm{Im}}}

\newcommand{\grad}{\mathop{\rm{grad}}}

\renewcommand{\d}{{\rm{d}}}

\newcommand{\e}{\varepsilon}
\newcommand{\m}{\setminus}
\newcommand{\fin}{\hspace*{\fill}$\Box$\vspace*{2mm}}

\newcommand{\cN}{{\mathcal N}}

\newcommand{\bR}{{\mathbb R}}
\newcommand{\bC}{{\mathbb C}}

\newcommand{\bN}{{\mathbb N}}

\renewcommand{\mathbb}{\mathsf}

\begin{document}

\title[Fibrations at infinity of real polynomial maps]{Real polynomial maps and singular open books at infinity}

%

\author{\sc R. N. Ara\'ujo dos Santos}
\address{ICMC,
Universidade de S\~ao Paulo,  Av. Trabalhador S\~ao-Carlense, 400 -
CP Box 668, 13560-970 S\~ao Carlos, S\~ao Paulo,  Brazil}
\email{rnonato@icmc.usp.br}

\author{Ying Chen}
\address{ICMC,
Universidade de S\~ao Paulo,  Av. Trabalhador S\~ao-Carlense, 400 -
CP Box 668, 13560-970 S\~ao Carlos, S\~ao Paulo,  Brazil}
\email{yingchen@icmc.usp.br}

\author{Mihai Tib\u ar}
\address{Math\'ematiques, UMR-CNRS 8524, Universit\'e Lille 1,
59655 Villeneuve d'Ascq, France.}
\email{tibar@math.univ-lille1.fr}



\thanks{Y. Chen acknowledges the Brazilian grant FAPESP-Proc. 2012/18957-7.
R. Ara\'ujo dos Santos acknowledges the grants USP-COFECUB Uc Ma 133/12, CNPq 474701/2012-3 and Fapesp 2011/20804-1. M. Tib\u ar acknowledges the support from USP-COFECUB grant Uc Ma 133/12.}


\begin{abstract}
We provide significant conditions under which we prove the existence of stable open book structures at infinity, i.e. on spheres $S^{m-1}_R$ of large enough radius $R$. We obtain new classes of real polynomial maps $\bR^m \to \bR^p$ which induce such structures.
\end{abstract}

\maketitle


\section{Introduction}\label{s:poly}

Open book structures produced on small spheres $S^{2n-1}_\e$ by holomorphic function germs $f : (\bC^n, 0) \to (\bC, 0)$  have their origins in  Milnor's work \cite{Mi}. Milnor actually shows that the map $f/|f| : S^{2n-1}_\e \m K_\e \to S_1^1$ is a locally trivial fibration without any condition upon the singular locus. In case of an isolated singularity at the origin, the link $K_\e := f^{-1}(0) \cap S^{2n-1}_\e$ becomes a fibre of a trivial fibration in a thin tube  $f_| : S^{2n-1}_\e \cap f^{-1}(D_\delta)\to D_\delta$, for small enough $\e \gg \delta >0$, which property is part of the classical definition of an ``open book structure'' cf \S \ref{ss:open}.

  If $f$ has nonisolated singularities, then a dotted tube fibration $f_| : B_\e \cap f^{-1}(D^*_\delta)\to D^*_\delta$ still exists due to the Thom regularity along Whitney strata of $f^{-1}(0)$ as embedded in the small ball $B_\e$,  as first proved by Hironaka \cite{Hi}. Together with Milnor's fibration of  $f/|f|$, this yields a so-called ``open book structure with singular binding'' on small enough spheres, cf \cite{ACT}, see Definition \ref{d:booksinginf}.

 For real analytic map germs $\psi : (\bR^m, 0) \to (\bR^p, 0)$, Milnor  proved in \cite[Theorem 11.2]{Mi} that if $\psi$ has \emph{isolated singularity}, i.e. $\Sing \psi = \{0\}$,  then there exists a locally trivial fibration
 $S^{m-1}_\e \m K_\e \to S_1^{p-1}$ independent of $\e > 0$ up to diffeomorphism type, even if the map $\psi/\|\psi\|$ fails in general to provide a locally trivial fibration unlike the above cited case of holomorphic functions. Together with the tube fibration defined by $\psi$ near the link $K_\e := \psi^{-1}(0) \cap S^{m-1}_\e$, these yield a so-called \emph{open book structure with codimension $p$ binding} on the sphere $S^{m-1}$, cf \cite{dST0}, \cite{dST1}.


We address here the similar existence problem for polynomial maps $\Psi : \bR^m \to \bR^p$, $m\ge p \ge 1$, and on arbitrarily large spheres.  Several papers focussed on estimating the set of bifurcation values $B(\Psi)\subset \bR^p$, more precisely those  produced by the asymptotic nonregular behaviour of fibres at infinity: first in case of holomorphic functions, e.g. \cite{Su}, \cite{Br}, and then in the general setting of real maps (\cite{Ra}, \cite{KOS}, \cite{CT}, \cite{DRT} etc). 
This new phenomenon has its impact on the existence of fibrations (see e.g. \cite{Ti}). In the simple example $f= x + x^2 y$ pointed out by Broughton \cite{Br}, the value $0$ is atypical without being singular.  Moreover, even in the total absence of bifurcation values, the map $\Psi/ \|\Psi\|$ may not induce a fibration on any large enough sphere, see Example \ref{e:nofib}.

In case of a complex polynomial $f: \bC^n \to \bC$, by using Milnor's method in the large scale, N\'emethi and Zaharia \cite{NZ} showed that the map $f/|f|$ is a locally trivial fibration on a large enough sphere $S^{2n-1}_R$ as soon as $f$ has at most the value $0$ as asymptotic $\rho$-nonregular value. 
Attaching to this fact the observation  that $f$ defines a locally trivial fibration near the link as a consequence of the Thom regularity along $V = f^{-1}(0)$, we get  an  \emph{open book decomposition at infinity} (i.e. on spheres of large enough radius $R$) \emph{with singular binding}, see Definition \ref{d:booksinginf}, Theorem \ref{t:proto}. 

Starting from this prototype, we produce here classes of genuine real maps which still induce open book structures at infinity, and in the same time we indicate the limits of such existence results. 
First, we show by  Theorem \ref{t:sing} that under condition (i) below, condition (ii) is necessary and sufficient for the existence of open books at infinity induced by the map $\frac{\Psi}{\|\Psi\|}$. Both conditions are defined in terms of the Milnor set $M(\cdot)$ introduced in Definition \ref{d:M}:\\
\\
 (i). $\overline{M(\Psi)\m V} \cap  V$ is bounded, \\
(ii). $M(\frac{\Psi}{\|\Psi\|})$ is bounded.\\

  Condition (i) expresses the transversality of the fibres of $\Psi$ different from $V := \Psi^{-1}(0)$ to any  sphere $S^{m-1}_R$ of large enough radius $R$ in some small enough neighbourhood $N$ of the link $K_R :=V \cap S^{m-1}_R$ and thus satisfies the requirement of Definition \ref{d:M} concerning $N$, whereas condition (ii) is just the 
transversality of the fibres of $\frac{\Psi}{\|\Psi\|}$ to such spheres, at any point outside the link. These conditions turn out to be independent and are related to those  in case of germs \cite{ACT}. 
We then find classes
of real polynomial maps which satisfy them, as follows.

We show that condition (i) is satisfied whenever: $\Sing(\Psi)\cap V$ is  bounded (Corollary \ref{t:nonsing}), or  $\Psi$ is a complex polynomial function $\bC^n \to \bC$ (Theorem \ref{t:proto}), or $\Psi$ is a polar weighted homogeneous mixed polynomial (Theorem \ref{t:homogen}).

It was known (cf \cite{NZ}, see Theorem \ref{t:proto}) that condition (ii) is satisfied in case of a semi-tame complex polynomial function $f$. We show here that condition (ii) is also satisfied in the following purely real settings:
$\Psi$ is real radial weighted homogeneous and $\Sing(\Psi)\subset V$ (Proposition \ref{p:radial}), or $f$ a semi-tame mixed polynomial as in Theorem \ref{t:semitame} and Corollary \ref{c:bp}, or  $\Psi$ is a polar weighted homogeneous mixed polynomial (Theorem \ref{t:homogen}), or has a certain Newton non-degeneracy property as pointed out in \S \ref{ss:newton}.

\medskip

Let us present here the an example (inspired from  \cite{oka1} and \cite{Ch-thesis}). More examples of real maps will be offered in \S \ref{exemples}.

 \begin{exam}
Let $f:\mathbb{C}^{2}\rightarrow\mathbb{C}$, $f(x,y)=(2x^{2}+|x|^{2})y$. The singular locus of $f$ is $\{(x,y)\in\mathbb{C}^{2}\mid x=0\}$ and  is contained in $V = \{ x=0\} \cup \{ y = 0\}$. By computing the set $M(f)\setminus\Sing(f)$ we find that $\overline{M(f)\setminus V}\cap V$ is bounded. Since $f$ is radial homogeneous, it follows from
 Proposition \ref{p:radial} and Theorem \ref{t:sing} that $f$ verifies the above conditions (i) and (ii) and thus, by Theorem \ref{t:sing}, has an open book structure  at infinity induced by $f/| f|$.
\end{exam}


\section{Fibrations and open book decompositions on spheres}\label{s:open}

\subsection{Higher open books with singular binding}\label{ss:open}
\begin{defn}\label{d:booksinginf}  (\cite{ACT}).
 We say that the pair $(K, \theta)$ is a \textit{higher open book structure with singular binding}  on an analytic manifold $M$ of dimension $m-1 \ge p \ge 2$ if
 $K\subset M$ is a singular real subvariety of codimension $p$ and $\theta : M\setminus K \to S^{p-1}_1$ is a locally trivial smooth fibration such that $K$ admits a neighbourhood $N$ for which the restriction
$\theta_{|N\setminus K}$ is the composition $N\setminus K  \stackrel{h}{\to} B^p \setminus \{ 0\} \stackrel{s/\|s\|}{\to} S^{p-1}_1$ where $h$ is a locally trivial fibration and $B^p$ denotes some open ball at the origin of $\bR^p$. 

 One says that the \textit{singular fibered link} $K$ is the \textit{binding}  and that the (closures of) the fibers of $\theta$ are the \textit{pages} of the \textit{open book}.
 \end{defn}

From the above definition it also follows that $\theta$ is surjective. 
In the classical definition of an open book (see e.g. \cite{Wi}, \cite{Et}) we have $p=2$, $K\subset M$
is a 2-codimensional submanifold which admits a neighbourhood $N$ diffeomorphic to $B^2 \times K$ 
for which $K$ is identified to $\{ 0\} \times K$ and the restriction  
$\theta_{|N \m K}$ is the following composition  $N\m K \stackrel{\diffeo}{\simeq} B^2 \m \{ 0\} \times K \stackrel{\proj}{\to} B^2 \m \{ 0\}  \stackrel{s/\|s\|}{\to} S^1$.  
In \cite{dST1} we have enlarged this definition to $p> 2$ (``higher'' open book structure), namely $K$ is a $p$-codimensional  submanifold.
 
Observe that the difference between the classical and the new definition \ref{d:booksinginf} is that $K$ is here a (singular) subvariety and that the  trivial fibration $N\setminus K  \to B^p \setminus \{ 0\}$ is replaced by a locally trivial fibration $h$.

\subsection{$\rho$-regularity}

Let $\rho$ denote some distance function on $\bR^m$.
The transversality of the fibres of a map $\psi$ to the levels of $\rho$ will be called  \textit{$\rho$-regularity}. This regularity is a basic tool, used by Milnor in the local setting, by Mather, Looijenga, Bekka etc in the stratified setting and for instance in \cite{NZ, Ti-reg}, \cite{Ti}, \cite{DRT} at infinity. This condition is used to produce locally trivial fibrations. This is for instance the essential ingredient in the proof that regular stratifications (Whitney (b)-regular or Bekka (c)-regular for instance) imply local triviality, roughly speaking.  In this paper we shall tacitly use as $\rho$ the Euclidean distance function and its open balls and spheres or radius $R$ denoted by $B_R$ and $S_R$ respectively. Let us give the general definition which fits in the germs setting as well as at infinity.

\begin{defn}\label{d:M}
Let $U\subset \bR^m$ be an open set and let $\rho : U \to \bR_{\ge 0}$ be a proper analytic function. 
We say that the set of \textit{$\rho$-nonregular points} (sometimes called \textit{the Milnor set}) of an analytic map  $\Psi :U \to \bR^p$ is the set of non-transversality between $\rho$ and $\Psi$:
 \[ M(\Psi) := \{ x\in U \mid \rho \not\pitchfork_x \Psi\}.\]

We call:
\[
 S(\Psi):=\{t_0\in\bR^p\mid\exists \{\mathrm{x}_j\}_{j\in \bN}\subset M(\Psi), \lim_{j\to\infty}\|\mathrm{x}_j\|=\infty\mbox{ and }\lim_{j\to\infty}\Psi(\mathrm{x}_j)=t_0\}.
\]
the \emph{set of asymptotic $\rho$-nonregular values}. 
 
 Similarly, the set of $\rho$-nonregular points of $\frac{\Psi}{\|\Psi\|} :U\m V \to S^{p-1}_1$ is the set:
 \[ M(\Psi/\|\Psi\|) :=  \{ x\in U \m V \mid \rho \not\pitchfork_x \Psi/ \|\Psi\| \}.\]
 \end{defn}
 
\begin{rem}
It follows from the above definition that $M(\Psi)$ is an analytic relatively closed subset of $U$ and it contains the singular set $\Sing \Psi$. As for $M(\frac{\Psi}{\|\Psi\|})$, it is semi-analytic but does no necessarily include $\Sing \Psi$. However, we clearly have $M(\frac{\Psi}{\|\Psi\|}) \subset M(\Psi)\m V$.
\end{rem}


\section{Open books with singular binding at infinity. General results}\label{ss:global}

\subsection{Milnor fibration induced by $\frac{\psi}{\|\psi \|}$}

We focus here on the existence of a locally trivial fibration:
 
 \begin{equation}\label{eq:milnormap}
 \frac{\Psi}{\|\Psi \|} : S^{m-1}_R \m K_R \to S_1^{p-1}.
\end{equation}

Even in the total absence of bifurcation values, the map $\Psi/ \|\Psi\|$ may not induce a fibration \eqref{eq:milnormap} on any large enough sphere, as we can see in the following example inspired from \cite{TZ}:

\begin{exam}\label{e:nofib}
Let $\Psi:\mathbb{R}^{3}\rightarrow\mathbb{R}^{2}$, $\Psi(x,y,z)=(y(2x^{2}y^{2}-9xy+12),z)$.
The bifurcation set $B(\Psi)$ is empty, but the map $\Psi/ \|\Psi\|$ has singularities on any sphere $S^{2}_R$ with $R\gg 1$. Indeed, from the definition we get that  $M(\Psi)=\{(x,y,z)\in\mathbb{R}^{3}\mid 4xy^{4}-9y^{3}-6x^{3}y^{2}+18x^{2}y-12x=0\}$. By considering sequances of points such that $x\to \infty$ and $xy\to 1^+$ we obtain that 
 $M(\Psi)$ has unbounded branches which are asymptotically tangent to fibres $\Psi^{-1}(0,z)$, for any $z\in R$.
Moreover, one can easily check that those asymptotical branches of $M(\Psi)$ are in $M(\frac{\Psi}{\left|\Psi \right|})$ too.  
\end{exam}

We prove the following criterion for the existence of an open book structure induced by $\Psi/ \|\Psi\|$.
For our zero locus $V = \Psi^{-1}(0)\subset \bR^m$, we say that ``$V$ has codimension $p$ at infinity'' if for any radius $R\gg 1$, $\codim_\bR V \setminus B_R = p$, in the sense that every irreducible component of $V$ has this property.

\begin{thm}\label{t:sing}
Let $\Psi : \bR^m \to \bR^p$ be a real polynomial map such that $\codim_\bR V = p$ at infinity and that $\overline{M(\Psi)\setminus V}\cap V$  is bounded. Then we have the equivalences:
 \begin{enumerate} 
\item $(K_R, \frac{\Psi}{\|\Psi\|})$ is an open book decomposition of $S_R^{m-1}$ with singular binding, independent of the high enough radius $R\gg 1$ up to $C^\infty$ isotopy.
\item $M(\frac{\Psi}{\|\Psi\|})$ is bounded.
\end{enumerate}
\end{thm}

\begin{proof}
\noindent
(a) $\Rightarrow$ (b). By assumption, $\frac{\Psi}{\|\Psi\|}$ is a locally trivial $C^\infty$-fibration and in particular induces a submersion $S_{R}^{m-1}\setminus K_{R}\to S^{p-1}$ for all large enough spheres $S_R^{m-1}$. This implies that $M(\frac{\Psi}{\|\Psi\|})$ is bounded.

\noindent
(b) $\Rightarrow$ (a).
If $\displaystyle{\overline{M(\Psi)\setminus V}\cap V}$  is bounded then there exist some $R_{0}\gg 1$ and an open neighborhood $\cN$ of $V\cap (\bR^m\setminus B^m_{R_0})$ depending on $R_0$ such that $M(\Psi) \cap (\cN \m V)$ is empty.

 Let us fix some $R\ge R_0$ and denote $V_R := V\m B_R$. 
The map $\displaystyle{\frac{\Psi}{\|\Psi\|}:S_{R}^{m-1}\setminus K_R\to S^{p-1}}$ is a submersion by hypothesis but one cannot apply Ehresmann's theorem since it is not proper if $K_R$ is not empty.
Nevertheless we may use the fibration structure of the map $\Psi$ near $V_R$ in order to control $\frac{\Psi}{\| \Psi\|}$ at the points of the link $K_R$. Indeed, we consider a tubular neighbourhood $N$ of $K_R$ included in $\cN \cap S_R$ and observe that the map  $\Psi : N\m K_R \to B_\delta^{p}\m \{ 0\}$ is a submersion. By taking a tubular neighbourhood $N$ which is also compact,  this map is proper too, hence a locally trivial fibration by Ehresmann's theorem.

Then, composing the map with the radial projection $B_\delta^{p}\m \{ 0\} \to S_1^{p-1}$ we get a proper submersion which moreover coincides with the map $\frac{\Psi}{\|\Psi\|}$.  This fits in the Definition \ref{d:booksinginf} and provides the desired open book decomposition with singular binding.
\end{proof}

\begin{rem}
The hypothesis of Theorem \ref{t:sing} includes the possibility of non-bounded branches of the set $\Sing(\Psi)\m V$. An open problem is to find such an example satisfying Theorem \ref{t:sing}.
\end{rem}

\subsection{Open books with smooth binding}
The hypothesis ``$\Sing(\Psi)\cap V$ is  bounded'' is the most general one under which $(K_R, \frac{\Psi}{\|\Psi\|})$ may yield an open book structure on $S_R^{m-1}$ with non-singular binding. We then have:
\begin{cor}\label{t:nonsing}
Let $\Psi : \bR^m \to \bR^p$ be a real polynomial map  such that $\codim_\bR V = p$ at infinity and let $\Sing(\psi)\cap V$ be bounded. Then
  $(K_R, \frac{\Psi}{\|\Psi\|})$ is an open book decomposition of $S_R$,  independent  of the radius $R\gg 1$ up to isotopy,  if and only if $M(\frac{\Psi}{\|\Psi\|})$ is bounded. 
\end{cor}
The proof follows from Theorem \ref{t:sing}, with the value $R_0$ as in its proof, via the following result which is an application of the curve selection lemma.
\begin{lem}\label{l:bounded}
If $\Sing(\Psi)\cap V$ is bounded then $\displaystyle{\overline{M(\Psi)\setminus V}\cap V}$ is bounded. \fin
\end{lem}

\subsection{Complex polynomials}\label{ss:complex} \ 
It has been proved by N\' emethi and Zaharia \cite{NZ} that for  complex polynomials $f$ with $S(f) \subset \{ 0\}$, called ``semi-tame'', the map $f/\| f\| : S_R^{2n-1} \m K_R \to S^1$ is a locally trivial fibration. 
We complete here this result by pointing out that semi-tame complex polynomials induce open book decompositions at infinity.  

\begin{thm}\label{t:proto}
 Let $f:\mathbb{C}^{n}\rightarrow\mathbb{C}$, $n\ge 2$, be a non-constant complex polynomial.
If $S(f)\subset \{ 0\}$, then $(K_R, \frac{f}{\|f \|})$ is an open book decomposition of $S_R^{2n-1}$ with singular binding, independent of the large enough radius $R\gg 1$ up to $C^\infty$ isotopy.
\end{thm}
\begin{proof}
 We just show that the conditions of our Theorem \ref{t:sing}  are satisfied. The hypersurface $V= f^{-1}(0)$ has indeed real codimension 2.  The proof of the condition "$\displaystyle{\overline{M(\Psi)\setminus V}\cap V}$ is bounded" can be traced back e.g. to \cite[Theorem 3.1.2]{Ti}, first part of its proof, and it is due to the Thom regularity condition along the Whitney strata of the zero locus of any analytic function, as shown by Hironaka \cite{Hi}, and later, with completely different types of proofs, by e.g. \L ojasiewicz,  Hamm and L\^e \cite{HL} or  Tib\u ar \cite[Theorem 2.9]{Ti-compo}. 

Condition (b) of Theorem \ref{t:sing} has been shown implicitly in the proof of \cite[Theorem 1]{NZ}.  
\end{proof}

\begin{rem}
In \cite[Theorem 10]{NZ} one also compares the fibre of the open book decomposition at infinity as in Theorem \ref{t:proto} to the general affine fibre of $f$. They are not diffeomorphic in general, unlike in the local setting treated by Milnor \cite{Mi}.

The condition $S(f) \subset \{ 0\}$ of the above theorem can be controlled for instance by the non-degeneracy of $f$ on all faces of its Newton polyhedron at infinity, including the so-called ``bad faces'', see the details in  \cite{NZ}.  This control may be extended to mixed polynomials, see  \cite{CT}, \cite{Ch-thesis} and \S \ref{ss:newton} at the end.

\end{rem}


\section{Classes of real maps with fibration properties}\label{exemples}

We find several significant classes of real analytic maps which have either property (i) or property (ii) from the Introduction. By combining these classes we produce examples of open book structures with singular binding.
We start with classes having property (ii): $M(\frac{\Psi}{\|\Psi\|})$ is bounded.

\subsection{Radial weighted homogeneous maps}\label{ss:radial} \ 

Let us consider the $\bR_+$-action on $\bR^m$:
$\rho\cdot x  =(\rho^{q_{1}}x_{1},\ldots,\rho^{q_{m}}x_{m})$ for $\rho \in \bR_+$ and $q_1, \ldots, q_m \in \bN^*$ relatively prime positive integers. If $\Psi (\rho \cdot x) = \rho^d\Psi(x)$ for all $x$ in the domain then we say that $\Psi$ is \textit{radial  weighted-homogeneous}.
\begin{pro}\label{p:radial}
  Let $\Psi$ be radial weighted-homogeneous and let   $\Sing \Psi \subset V$. Then $M(\frac{\Psi}{\|\Psi\|})$ is bounded.
\end{pro}
\begin{proof} 
The proof goes along the lines of \cite[Proposition 3.2]{ACT}.
  Let $\gamma (x) := \sum_{j=1}^m q_j x_j \frac{\partial}{\partial x_j}$ be the corresponding Euler vector field on $\bR^m$; we have $\gamma (x) = 0$ if and only if $x=0$.  
Suppose that there exists $x\notin V$ such that $x\in M(\frac{\Psi}{\|\Psi\|})$. Then we have $\rank\Omega_\psi(x)\leq p$
where $\Omega_\psi(x)$ denotes the $[(p-1)p/2 +1] \times m$ matrix 
 having on each row a vector $\omega_{i,j}(x) := \Psi_i(x)\hspace{1pt} \grad \Psi_j(x) - \Psi_j(x)\hspace{1pt} \grad \Psi_i(x)$, for $i,j = 1, \ldots ,p$ with $i<j$, except of the last row which contains the position vector  $(x_1, \ldots , x_m)$.
Observing that $\langle \gamma(x),  \grad \Psi_i(x) \rangle = d \cdot \Psi_i(x)$ for any $i$
we have:
\[ \langle \gamma(x),  \omega_{i,j}(x) \rangle = d [\Psi_i(x)\Psi_j(x) - \Psi_j(x)\Psi_i(x)] = 0,
 \]
which means that the Euler vector field $\gamma(x)$ is tangent to the fibres of $\frac{\Psi}{\|\Psi\|}$.
We also have: $\langle \gamma(x),  x \rangle = \sum_i q_i x_i^2 > 0$,
for $x \not= 0$, which means that $x\notin M(\frac{\Psi}{\|\Psi\|})$ and therefore contradicts our assumption. 
\end{proof}

\medskip

\subsection{Mixed polynomials}\label{ss:semitame}\

Let $f:\mathbb{C}^{n}\rightarrow\mathbb{C}$ be a non-constant mixed polynomial. This is by definition a polynomial in variables $z$ and $\bar z$ (cf \cite{oka1}, \cite{oka2}), namely:
  \[
   f(\mathbf{z},\overline{\mathbf{z}})=\sum_{\nu,\mu}c_{v,\mu}\mathbf{z}^{\nu}\mathbf{\overline{z}}^{\mu}\label{eq:mixed}
   \]
   where $c_{v,\mu}\neq0$, $\mathbf{z}^{\nu}:=z_{1}^{v_{1}}\cdots z_{n}^{v_{n}}$ and $\mathbf{\overline{z}}^{\mu}:=\overline{z}_{1}^{\mu_{1}}\cdots \overline{z}_{n}^{\mu_{n}}$ for n-tuples $v=(v_{1},\ldots,v_{n})$, $\mu=(\mu_{1},\ldots,\mu_{n})\in\mathbb{N}^{n}$. In  particular a complex polynomial is a mixed polynomial in which the conjugates $\bar z$ do not occur.
    \\
   For a mixed polynomial $f$, we shall use the derivation with respect to $\mathbf{z}$ and $\overline{\mathbf{z}}$ as follows:
   \[
  \mathrm{d}f:=\left(\frac{\partial f}{\partial z_{1}},\cdots,\frac{\partial f}{\partial z_{n}}\right),\overline{\mathrm{d}}f:=\left(\frac{\partial f}{\partial\overline{z}_{1}},\cdots,\frac{\partial f}{\partial\overline{z}_{n}}\right)
   \]
  The Milnor set of a mixed polynomial $f$ appears to be the following (cf \cite{CT}):
  \[ M(f)=\left\{ \mathbf{z}\in\mathbb{C}^{n}\mid\exists\lambda\in\mathbb{R} \mbox{ and } \mu\in\mathbb{C}^{*}, \mbox{ such that }\lambda\mathbf{z}=\mu\overline{\mathrm{d}f}(\mathbf{z},\overline{\mathbf{z}})+\overline{\mu}\overline{\mathrm{d}}f(\mathbf{z},\overline{\mathbf{z}})\right\} .\]
  In particular, for $\lambda=0$ we get the singular locus of $f$.

\begin{thm}\label{t:semitame}
 Let $f$ be a mixed polynomial with $S(f) \subset \{ 0\}$ and such that

\noindent
{\rm (*)} the image $f(\Sing (f/| f |))$ is bounded and 0 is isolated point in $f(\Sing (f/| f |))$.

Then $M(f/| f |)$ is bounded.
 
\end{thm}

\begin{rem}\label{r:**}
 Since $\Sing (f/| f |)\subset \Sing f$,   condition (*) is implied in particular by the condition: \\
 (**) {\it the set of critical values $f(\Sing f)$ is bounded and  0 is an isolated critical value.}

Note that (**) is automatically true for holomorphic $f$ and even for meromorphic functions $h/g$ on the complement of the the zero locus $hg =0$ since $h/g$ is  holomorphic on this complement. 
\end{rem}

\begin{proof}
 Suppose $S(f) \subset \{ 0\}$. By using the curve selection lemma at infinity for the semi-algebraic set $M(f/| f |)$, we consider a path $p(t)\in M(f/| f |)$ tending to infinity as $t\to 0$. Under the assumptions of the theorem, we want to reach a contradiction. 
 Start by writing the expansions like Milnor did in the local setting \cite{Mi} (see \cite{NZ} for complex polynomials). Instead of the \emph{gradient} of a holomorphic function; this will be the object called $\nabla$ below:

\[ \begin{array}{lc}
 p(t) = a t^\alpha + \cdots, &  \mbox{ where } \alpha <0, \ a\not= 0,\\
 f(p(t)) = b t^\beta + \cdots, &  \mbox{ where } \  b\not= 0,\\
\nabla(p(t)) := i ( \bar f \bar \d f - f \overline{\d f})(p(t)) = c t^\gamma + \cdots
\end{array}
\]
We moreover have $\beta \not= 0$, since if $\beta = 0$ then $b\in S(f)$ and, since $b\not= 0$, we get a contradiction with our assumption $S(f) \subset \{ 0\}$.

The condition $p(t)\in M(f/| f |)$ is expressed by the following equation (see \cite{Ch} and \cite[Def. 5.2.2]{Ch-thesis}):
\begin{equation}\label{eq:M}
 \lambda(t) p(t) = \nabla(p(t)), \ \ \mbox{ with } \  \lambda(t) \in \bR.
\end{equation}

We have: $\nabla(p(t)) \equiv 0$ if and only if, either (a). $p(t) \in \Sing(f/| f |)$ and $f(p(t)) \not\equiv 0$, or (b). $f(p(t)) \equiv 0$.

The case (a) is excluded by the hypotheses ``semi-tame'' and (*) of the theorem.
The case (b) has no impact on the set $M (f/| f |)$ since this is by definition disjoint from $f=0$. Actually the case (b) occurs if and only if the zero locus $V$ of $f$ contains non-compact components of $\Sing f$, which is the case in many examples.

We may therefore assume from now on that $c\not= 0$. One has to express $\frac{\d f (p(t))}{ \d t}$ and  like done in \cite{Ch-thesis} and \cite{Ch},  take the real parts $\re(\cdot)$ and find that:

\[
 \re ( i\bar f \frac{\d f}{ \d t}) = \re \langle \frac{\d p}{ \d t}, \nabla(p(t))\rangle
\]

Using the expansions and relation \eqref{eq:M}, we get:
\begin{equation}\label{eq:expansion}
 \re ( i | b|^2 \beta t^{2\beta -1} + \cdots ) = \re ( ( \alpha \frac{c}{a}| a|^2 + i v) t^{\alpha + \gamma -1}+ \cdots )
\end{equation}

where $v$ is some real constant. 
Since $\beta \not= 0$, both leading coefficients of \eqref{eq:expansion} are non-zero. However, left side is purely imaginary and right side has a non-zero real part, which is a  contradiction. Our proof is finished.
\end{proof}
\begin{rem}
  We have seen that the assumption (*) is the major difference to the 
holomorphic or meromorphic settings. It  is indeed possible that for mixed functions we do not have fibration exactly because of this condition. This is the case in the example 
due to Oka and computed in \cite[Example 5.3.3]{Ch-thesis}, where we have $S(f) = \emptyset$ but no fibration of $f/| f |$ due actually to the failure of the condition (**) of Remark \ref{r:**}.

\end{rem}

\begin{exam}
$f(x,y) = x(1+ \bar x y)$ is a mixed version of Broughton's example.
We easily compute that  $\Sing f = \emptyset$, $V$ has two components of real codimension $2$,  $M(f) \cap V$ is bounded and $S(f) = \{ 0\}$.
This verifies the hypotheses of Theorem \ref{t:semitame} and Corollary \ref{t:nonsing}, hence $f$ yields an open book structure at infinity with smooth binding. 
 
\end{exam}

\subsection{Mixed polynomials of type $f = g\bar h$}\ 

Whenever $f = g\bar h$, for some complex holomorphic functions $g$ and $h$, it turns out that the ``semi-tame'' condition imposed to the rational map $g/h$ is sufficient to insure the fibration of $g\bar h/ | g\bar h |$, which fact was observed by Bodin and Pichon \cite{BP}. We show how this fact follows from our Theorem \ref{t:semitame}. 
\begin{cor} \label{c:bp} 
 Let $S(g/h) \subset \{ 0\}$. Then $M(g\bar h/ | g\bar h |)$ is bounded.
\end{cor}
\begin{proof}
  On the complement of the hypersurface $V := \{gh =0\}$ we have the equality
 $g\bar h/ | g\bar h |  = \frac{g/h}{|g/h|}$
which  implies: 
$ \Sing (g\bar h/ | g\bar h | )\m V = \Sing (\frac{g/h}{|g/h|})\m V$
and $M(g\bar h/ | g\bar h |) \m V =M(\frac{g/h}{|g/h|}) \m V $.
The proof of our Theorem \ref{t:semitame} applies to meromorphic functions too. The  conditions (*) and  (**) are fulfilled since  $g/h$ is holomorphic on the complement of $V$ and so $\Sing (g/h) \setminus V$ is contained in finitely many fibres and we have the equality $\Sing (g/h)\m V= \Sing(\frac{g/h}{|g/h|})\m V$. Therefore our theorem shows that $M(\frac{g/h}{|g/h|}) $ is bounded. It then follows from the above equality that $M(g\bar h/ | g\bar h |)$ is also bounded.
\end{proof}
To finish the proof that the map $g\bar h$ actually induces a Milnor fibration at infinity \eqref{eq:milnormap} one follows the proof by Milnor in the local setting \cite[Lemmas 4.6, 4.7, Theorem 4.8]{Mi} and combine with the proof of our Theorem 
\ref{t:semitame} above. Let us also notice that the map $g\bar h$ does not necessarily have the Thom regularity property 
at $V$ (see Example \ref{ex:thom} and \cite{ACT}) thus the property (i) of the Introduction may not be fulfilled.

\medskip

In the remainder of the section we single out classes of mixed functions where both conditions (i) and (ii) from the Introduction are satisfied and therefore yield open book structures with singular binding at infinity by Theorem \ref{t:sing}. Let us begin with an example.

\begin{exam}\label{ex:bounded} 
Let $f:\mathbb{C}^{2}\rightarrow\mathbb{C}$, $f(x,y)=(2+x)x^{2}y\overline{x}$.
We have $\Sing f=\{(x,y)\in\mathbb{C}^{2}\mid x=0\}\bigcup\{(-2,0)\} \subset V$.
By computations we get that $\overline{M(f)\setminus V}\cap V$ is bounded.
Theorem \ref{t:semitame} tells that  $M(\frac{f}{\left|f\right|})$ is bounded (and by direct computations  we see that it is actually empty).
It follows that $f/\| f\|$ induces an open book decomposition at infinity with singular binding.
\end{exam}

\subsection{Mixed polynomials of polar weighted homogeneous type}\ 
  
\begin{defn}\label{d:polar}
A mixed polynomial $f:\mathbb{C}^{n}\rightarrow\mathbb{C}$ is called
\emph{polar weighted homogeneous} if there exist $n$ integers $p_{1},\ldots,p_{n}$ with $\gcd(p_{1},\ldots,p_{n})=1$
and a positive integer $m_{p}$ such that $\sum_{j=1}^{n}p_{j}(v_{j}-\mu_{j})= k$
for every n-tuples $\nu$ and $\mu$. We call $(p_{1},\ldots,p_{n})$ the polar weight of $f$
and $k$ the polar degree of $f$.
\end{defn}

 \begin{thm}\label{t:homogen}
\sloppy Let $f:\mathbb{C}^{n}\rightarrow\mathbb{C}$ be a non-constant mixed polynomial which is polar weighted-homogeneous, $n\ge 2$, such that $\codim_\bR V = 2$ at infinity.
 Then $(K_R, \frac{f}{\|f \|})$ is an open book structure with singular binding on $S^{2n-1}_R$,
   independent (up to isotopies) of $R \gg 0$. 
\end{thm}

\begin{proof}
 We apply Theorem \ref{t:sing} and therefore we have to check here that its hypotheses are fulfilled, which means conditions (i) and (ii) from the Introduction.
By the codimension condition at infinity, for any large enough radius $R$,  the image of the restriction $f_{|S^{2n-1}_R}$ contains the germ at $0$ of a semi-analytic set of positive dimension.
  The $S^1$-action provided by the polar weighted homogeneous polynomial $f$ yields that if $a\in \im(f_{|S^{2n-1}_R})$ then $\lambda^k a\in \im(f_{|S^{2n-1}_R})$ for any $\lambda \in S^1$. It then follows that $\im(f_{|S^{2n-1}_R})$ contains some disk centered at the origin, for any $R\gg 1$ (where the radius of the disk depends of $R$).
 
Denote by $Z$ the closure of the real algebraic set $M(f)\m V$. Then $Z$ and $Z \cap V$ are closed semi-algebraic sets.
In these notations, condition (i)  has  the following obvious reformulation:\ \ 
(i'). \  $Z \cap V$ is bounded.     

If $Z \cap V$ is not bounded, then $Z \cap V$ contains an unbounded
maximal semi-algebraic  subset $A$ of positive dimension $s>0$.
 We may assume that $A$ is connected, otherwise we pick up a component. Then $Z$ contains a connected semi-algebraic subset M of dimension $\ge s+1$
such that $M\cap V = A$.  (This follows for instance from the fact that
semi-algebraic sets can be triangulated.) Then for almost any point $x \in A$, the sphere
$S^{2n-1}_R$ which passes through $x$ intersects $M$
and the germ at $x$ of this intersection, let it be denoted by $\Gamma_R$, has positive dimension.  (This also
follows from the triangulated structure of $M$). Then the image $f(\Gamma_R)$ is the germ at 0 of a semi-algebraic set of positive dimension.   
It then follows from the Curve Selection Lemma that $f(\Gamma_R)$ contains some germ at 0 of a real curve. This means that the set of critical values of the map $f_{|S^{2n-1}_R}$ contains points of any small enough radius.
For any such non-zero critical value, the $S^1$-action provided by the polar action tells that the whole circle of the corresponding radius is contained in the set of critical values of $f_{|S^{2n-1}_R}$.
 This shows that for any $R\gg 1$ there exists some small punctured disk centered at 0 (of radius depending on $R$) which is included in the set of critical values of the map $f_{|S^{2n-1}_R}$. But this contradicts the fact that the set of regular values must be dense. We have thus proved (i), namely that  $\overline{M(f)\setminus V}\cap V$  is bounded.

By a similar argument applied to the map $f/|f|$ instead of $f$, we get that $M(f/|f|)$ is bounded too. Then Theorem \ref{t:sing} applies and proves our claim.
\end{proof}

 \begin{exam}\label{ex:thom}
The mixed polynomial function $f=  | z_1 |^2 z_2$ is of type $g\bar h$ but not Thom regular along $\Sing f = \{ z_1 = 0\}\subset V$, see \cite{ACT}.
However it satisfies the hypotheses of Theorem \ref{t:homogen}. We may easily find such examples in more variables, such as: $f=  (z_1 + z_3^5) \bar z_1 z_2$.  They have the interesting property that they define open book structures at infinity  (by Theorem \ref{t:homogen}) without being Thom regular at V as holomorphic functions always are (see \S \ref{ss:complex}). 

 \end{exam}
 
\subsection{Mixed polynomials and Newton non-degeneracy at infinity}\label{ss:newton}

 The semi-tameness condition is not easy to control since it is an asymptotic condition. 
In \cite{CT},  \cite{Ch-thesis} and \cite{Ch} we have adapted the notion of Newton non-degeneracy to mixed polynomials and, more generally, to maps (real, mixed or complex) in \cite{CDTT}.  One has for instance the following result \cite[Theorem 3.5]{Ch}: \emph{if $f$ is Newton strongly non-degenerated on every face of the convex hull of its support, then $M (f/| f |)$ is bounded}. This is an extension of a result by N\'emethi-Zaharia \cite[p. 329-330]{NZ} for complex polynomials.
Moreover, \cite[Theorem 3.6]{Ch} shows that in order to prove condition (ii) it is enough to show that the critical loci of $\frac{f_{\triangle}}{|f_{\triangle}|}$ are empty on $\mathbb{C}^{*n}$ for any strictly bad face $\triangle$ of $\overline{\mathrm{supp}(f)}$, see the next example.

For  holomorphic polynomials of two variables, the condition (ii): ``$M(\frac{f}{|f|})$ is bounded'' is equivalent to the condition  $S(f)\subset \{0\}$, cf \cite{Bo}. In the mixed case, the following example shows that ``semi-tame'' is not a necessary condition for the existence of the Milnor fibration of $\frac{f}{|f|}$.
\begin{exam}\label{ex:non semitame}
Consider  $f(z_{1},z_{2})=z_{1}(1+\left|z_{2}\right|^{2}+z_{1}z_{2}^{4})$, taken from \cite{CT}.
We have $\Sing f=\emptyset$ and $B(f)=S(f)=\left\{ c\in\mathbb{C}\mid\left|c\right|=\frac{1}{4}\right\}$.
By Lemma \ref{l:bounded}, it follows that $\overline{M(f)\setminus V}\cap V$ is bounded. One can show that  $f$ is strongly non-degenerate.
Consider next the restriction $f_{\triangle}=z_{1}(\left|z_{2}\right|^{2}+z_{1}z_{2}^{4})$ of $f$ to the single ``strictly bad'' face $\triangle$ of $\overline{\mathrm{supp}(f)}$. By direct computations we find that $\Sing\frac{f_{\triangle}}{|f_{\triangle}|}\bigcap\mathbb{C}^{*2}=\emptyset$.
Then \cite[Theorem 3.6]{Ch} shows that $M(\frac{f}{|f|})$ is bounded.
Consequently, our Theorem \ref{t:sing} tells that $\frac{f}{|f|}$ induces an open book structure at infinity. 
\end{exam}


\end{document}